\documentclass[a4paper,12pt,reqno]{amsart}
\pdfoutput=1

\usepackage{mathrsfs,mathdots,amssymb}
\usepackage[T1]{fontenc}
\usepackage[utf8]{inputenc}
\usepackage[centering]{geometry}

\usepackage{microtype}

\usepackage{hyperref}
\theoremstyle{plain}

\newtheorem{theorem}{Theorem}[section]
\newtheorem*{theorem*}{Theorem}

\newtheorem{lemma}[theorem]{Lemma}
\newtheorem{proposition}[theorem]{Proposition}
\newtheorem{example}[theorem]{Example}
\numberwithin{equation}{section}

\theoremstyle{remark}
\newtheorem{claim}[theorem]{Claim}
\newtheorem{definition}[theorem]{Definition}
\newtheorem{remark}[theorem]{Remark}

 \def\C{\mathbb C}
   
\def \N{\mathbb N} 
\def\i{{\rm i}}

\def\wtilde{\widetilde} \def\what{\widehat}

\newcommand{\coloneqq}{\mathrel{\mathop:}=}

\newcommand{\R}{\mathbb{R}}
\newcommand{\In}{\operatorname{In}}

\newcommand{\rank}{\operatorname{rank}}

\title[Hurwitz stability criteria for matrix polynomials]%
{On generalization of classical Hurwitz stability criteria for matrix polynomials}
\date{29 September 2019}


\author[X. Zhan]{Xuzhou Zhan}
\address[X. Zhan]{School of Mathematical Sciences, South China Normal University,
    Guangzhou 510631, China}
\email{zhan@math.uni-leipzig.de}

\author[A. Dyachenko]{Alexander Dyachenko}
\address[A. Dyachenko]{Department of Mathematics, University College London,
    Gower St., London WC1E~6BT, UK}
\email{diachenko@sfedu.ru}

\begin{document}

\begin{abstract}
    In this paper, we associate a class of Hurwitz matrix polynomials with Stieltjes positive
    definite matrix sequences. This connection leads to an extension of two classical criteria
    of Hurwitz stability for real polynomials to matrix polynomials: tests for Hurwitz stability
    via positive definiteness of block-Hankel matrices built from matricial Markov parameters
    and via matricial Stieltjes continued fractions. We obtain further conditions for Hurwitz
    stability in terms of block-Hankel minors and quasiminors, which may be viewed as a weak
    version of the total positivity criterion.\\[.3em]


\end{abstract}

\maketitle

\section{Introduction}\label{SectionIntro}
Consider a high-order differential system
\begin{equation*}
    A_{0} y^{(n)}(t)+A_{1} y^{(n-1)}(t)+\cdots+A_{n} y(t)=u(t)
    ,
\end{equation*}
where $A_0,\dots,A_n$ are complex matrices, $y(t)$ is the output vector and~$u(t)$ denotes the
control input vector. The asymptotic stability of such a system is determined by the Hurwitz
stability of its characteristic matrix polynomial
\begin{equation*}
F(z)=A_{0} z^{n}+A_{1} z^{n-1}+\cdots+A_{n},
\end{equation*}
or to say, by that all roots of~$\det F(z)$ lie in the open left half-plane~$\Re z<0$. Many
algebraic techniques are developed for testing the Hurwitz stability of matrix polynomials,
which allow to avoid computing the determinant and zeros: LMI approach
\cite{HAP,HAPS,LPJ,LPJL}, the Anderson-Jury Bezoutian \cite{LRT, LT82}, matrix Cauchy indices
\cite{BA}, lossless positive real property~\cite{AB}, block Hurwitz matrix~\cite{KMS}, extended
Routh-Hurwitz array \cite{Gal}, argument principle \cite{HH} and so on.  However, the authors are unaware
of a suitable extension to the following classical landmarks which one can inevitably encounter
in the abundant study of the scalar case.

Gantmacher's monograph \cite{Gan} gives a comprehensive overview of issues related to Hurwitz
polynomials, among which is a characterization of the Hurwitz stability through the
corresponding Stieltjes continued fraction \cite[Theorem 16, Chapter XV, p232]{Gan}:

\begin{theorem*}[\textbf{Stability criterion via continued fractions}]
    A real polynomial $f(z)$ of degree~$n=2m$ or~$n=2m+1$ is a Hurwitz polynomial if and only if
    its even part $f_{e}(z)$ and its odd part $f_{o}(z)$, where~$f_e(z^2)+z f_o(z^2)=f(z)$,
    admit the following Stieltjes continued fraction:
    \begin{equation}\label{ScalarContinuedFraction}
        \frac{f_{o}(z)}{f_{e}(z)}=c_{0}+\cfrac{1}{z{c}_1+\cfrac{1}{c_{2m-2}+\cfrac{\ddots}{zc_{2m-1}+c_{2m}^{-1}}}},
    \end{equation}
    where ${c}_k>0$ for $k=1,\ldots, n$ and, when~$n=2m+1$, additionally~$c_{0}>0$.
\end{theorem*}

Theorem~17 of~\cite[Chapter XV]{Gan} connects Hurwitz polynomials with positive definite Hankel
matrices built from Markov parameters:

\begin{theorem*}[\textbf{Stability criterion via Markov parameters}]
    Given a real polynomial $f(z)$ of degree $n=2m$ or $n=2m+1$, define its Markov
    parameters~$(s_k)_{k=-1}^{\infty}$ from the rational function
    $$
    \frac{f_{o}(z)}{f_{e}(z)}=s_{-1}+\sum_{k=0}^{\infty} \frac{(-1)^k s_{k}}{z^{k+1}},
    $$
    where $f_e(z)$ and $f_o(z)$ are as above. Then $f(z)$ is a Hurwitz polynomial if and only if
    both Hankel matrices $[s_{j+k}]_{j,k=0}^{m-1}$ and $[s_{j+k+1}]_{j,k=0}^{m-1}$ are positive
    definite and, for~$n=2m+1$, additionally $s_{-1}>0$.
\end{theorem*}

The positive definiteness of the pair of matrices~$[s_{j+k}]_{j,k=0}^{m-1}$ and
$[s_{j+k+1}]_{j,k=0}^{m-1}$ turns to be equivalent to the total positivity of rank~$m-1$ of the
infinite Hankel matrix $[s_{j+k}]_{j,k=0}^{\infty}$: i.e., to that all minors of
$[s_{j+k}]_{j,k=0}^{\infty}$ of order $\leq m-1$ are positive and all minors of order~$\ge m$
are zero. Theorem 20 of~\cite[Chapter XV]{Gan} gives the corresponding alternative criterion of
Hurwitz stability: a real polynomial $f(z)$ is a Hurwitz polynomial iff
$[s_{j+k}]_{j,k=0}^{\infty}$ is totally positive and, for~$n=2m+1$, additionally~$s_{-1}>0$.
  
Some obstacles occur in testing Hurwitz stability of matrix polynomials when one tries to adopt
these classical tools. Such features as noncommutativity of the matrix product induce
limitations of the determinant approach to algebraic constructs involving coefficients of matrix
polynomials. The applicability of scalar methods to the matrix case is also strongly influenced
by lack of suitable correlations between the matrix coefficients of a matrix polynomial (even of
lower degree) and its zeros.

We follow some alternative lines to deal with the matrix extension. Based on the spectral theory
of matrix polynomials, the papers~\cite{LRT, LT82} solve some zero-separation problems for
matrix polynomials in terms of Anderson-Jury Bezoutian. Converting the solution from
\cite{LRT, LT82} into the form of matricial Markov parameters is the key step for our matricial
refinement of the stability criterion via Markov parameters.

The so-called matrix Hurwitz type polynomials, which are defined via matricial analogue of
Stieltjes continued fraction as in~\eqref{ScalarContinuedFraction}, are studied by Choque
Rivero~\cite{Riv} in connection to matricial Stieltjes moment problem. To extend to the matrix
case the stability criterion via continued fractions or, in other words, to uncover the relation
between matrix Hurwitz type polynomials and Hurwitz matrix polynomials, we link the Markov
parameters with the properties of Stieltjes moment matrix sequences.

We conclude the introduction with the outline of the paper.
Section~\ref{SectionHurwitzStieltjes} brings forth two natural questions concerning the
extension of the stability criterion via Markov parameters and via Stieltjes continued
fractions. It turns out that to give such extensions for all complex/real matrix polynomials is
impossible. In Section~\ref{SectionRouth-HurwitzMP}, we derive an inertia representation for
matrix polynomials in terms of the matricial Markov parameters. Our main results are provided in
Section~\ref{Main}. We deal with a relationship between Hurwitz matrix polynomials and a special
important type of matricial Stieltjes moment sequences. It leads us to matricial extensions of
the stability criteria via Markov parameters and via continued fractions to a special class of
matrix polynomials. Further conditions for Hurwitz stability are obtained in terms of Hankel
minors and Hankel quasiminors built from matricial Markov parameters. Unfortunately, the
block-Hankel total positivity is not generally guaranteed for Hurwitz matrix polynomials.

\section{Matricial Markov parameters, Stieltjes continued fractions and Hurwitz stability}\label{SectionHurwitzStieltjes}
We begin with some basic notation. Denote by~$\C$, $\R$, $\N_0$ and~$\N$, respectively, the sets
of all complex, real, nonnegative integer, and positive integer numbers. Unless explicitly
noted, we assume in this paper that $p,q\in \N$. Let~$\mathbb C^{p\times q}$ stand for the set
of all complex~$p\times q$ matrices. Let also~$0_p$ and~$I_p$ be, respectively, the zero and the
identity~$p\times p$ matrices. Given a matrix~$A$ we denote its transpose by~$A^{\rm T}$, and
its conjugate transpose by~$A^*$. If~$A\in \mathbb C^{p\times p}$, then we write~$A\succ 0$ if
it is positive definite, and $A\succeq 0$ if~$A$ is nonnegative definite.

We denote by $\mathbb C[z]^{p\times p}$ the set of $p\times p$ matrix polynomials, that is the
ring of polynomials in~$z$ with coefficients from~$\mathbb C^{p\times p}$; in
particular~$\mathbb C[z]=\mathbb C[z]^{1\times 1}$. So, each~$F(z)\in\mathbb C[z]^{p\times p}$
may be written as
\begin{equation}\label{ReprentationMatrixPolynomials}
    F(z)=\sum_{k=0}^{n} A_k z^{n-k},\quad \mbox{with }A_0,\ldots, A_n\in \mathbb C^{p\times p},
    \quad A_0\ne 0_{p}
\end{equation}
for certain~$n\in\N_0$, which is called the degree of~$F(z)$. In particular, $F(z)$ is
\emph{monic} if~$A_0=I_p$. Furthermore,~$F(z)$ may be represented as a $p\times p$ matrix whose
entries are scalar polynomials in~$z$, and~$\deg F$ then equals the maximal degree of the
entries. In most cases, the degree will be assumed to be at least two, since the results for
linear polynomials are trivial.

Given a matrix polynomial~$F(z)\in\mathbb C[z]^{p\times p}$ written as
in~\eqref{ReprentationMatrixPolynomials}, define $F^{\vee}(z)\in\mathbb C[z]^{p\times p}$ by
\begin{equation*}
F^{\vee}(z)\coloneqq\sum_{k=0}^n A^*_k z^{n-k}.
\end{equation*}

A matrix polynomial $F(z)\in\mathbb C[z]^{p\times p}$ is said to be \emph{regular} if
$\det F(z)$ is not identically zero. Clear that all monic matrix polynomials are regular. Given
a regular matrix polynomial $F(z)$, we say that $\lambda\in \mathbb C$ is a \emph{zero} of
$F(z)$ if $\det F(\lambda)=0$. Its \emph{multiplicity} is the multiplicity of $\lambda$ as a
zero of~$\det F(z)$. The spectrum~$\sigma(F)$ of~$F(z)$ is the set of all zeros of~$F(z)$.

\begin{definition}
    A matrix polynomial~$F(z)\in \mathbb C[z]^{p\times p}$ may be split into the
    \emph{even part}~$F_{e}(z)$ and the \emph{odd part}~$F_{o}(z)$ so that
    \[ 
        F(z)=F_{e}(z^2)+zF_{o}(z^2).
    \]
    For~$F(z)$ of degree~$n$ written as in~\eqref{ReprentationMatrixPolynomials},
    they are defined by
    \begin{equation*}
  F_{e}(z)\coloneqq \sum_{k=0}^m A_{2k}z^{m-k}
  \quad\text{and}\quad
  F_{o}(z)\coloneqq \sum_{k=1}^m A_{2k-1}z^{m-k}
\end{equation*}
when~$n=2m$, and by
\begin{equation*}
    F_{e}(z)\coloneqq \sum_{k=0}^m A_{2k+1}z^{m-k}
    \quad\text{and}\quad
    F_{o}(z)\coloneqq \sum_{k=0}^m A_{2k}z^{m-k}
\end{equation*}
when~$n=2m+1$.
\end{definition}

\begin{definition}\label{DefMarkovParameters}
    Let $F(z)\in \mathbb C[z]^{p\times p}$ be a monic matrix polynomial with the even part~$F_{e}(z)$
    and the odd part~$F_{o}(z)$.
\begin{enumerate}
\item[{\rm (i)}] In the even case~$n=2m$, suppose that
    \begin{equation}\label{Ge}
        G_1(z)\coloneqq F_{o}(z)\cdot (F_{e}(z))^{-1}\quad (\mbox{resp.\ }\ G_1(z)\coloneqq (F_{e}(z))^{-1}\cdot F_{o}(z)\,)
    \end{equation}
    admits the Laurent expansion
    \begin{equation}
        G_1(z)=\sum_{k=0}^{\infty} (-1)^{k}z^{-(k+1)}{\bf s}_k \label{MakrovPE} 
    \end{equation}
    for large enough~$z\in \mathbb C$. If so, we call the matrix
    sequence~$({\bf s}_k)_{k=0}^{\infty}$ the \emph{sequence of right} (resp.\ \emph{left})
    \emph{Markov parameters} of $F(z)$.

\item[{\rm (ii)}] In the odd case $n=2m+1$, we may also consider~$G_1(z)$ as in~\eqref{Ge}
    if~$F_e(z)$ is regular. Its Laurent expansion then slightly differs:
    \begin{equation}\label{MakrovPO2}
        G_1(z)={{\bf s}_{-1}}+\sum_{k=0}^{\infty} (-1)^kz^{-(k+1)} {\bf s}_{k}
    \end{equation}
    for large enough~$z\in \mathbb C$. In this case, the matrix
    sequence~$({\bf s}_k)_{k=-1}^{\infty}$ is called the \emph{sequence of right} (resp.
    \emph{left}) \emph{Markov parameters of first type} of $F(z)$.

\item[{\rm (iii)}] Another option for the odd case $n=2m+1$, is to consider
    \[
        G_2(z)\coloneqq F_{e}(z)\cdot (F_{o}(z))^{-1}
        \quad (\mbox{resp.\ }\ G_2(z)\coloneqq (F_{o}(z))^{-1}\cdot F_{e}(z)\,)
    \]
    with the following Laurent representation:
    \begin{equation}
        G_2(z)=\sum_{k=0}^{\infty} (-1)^{k}z^{-k}{\bf s}_{k} \label{MakrovPO1}
    \end{equation}
    for large enough $z\in \mathbb C$. The matrix sequence $({\bf s}_k)_{k=0}^{\infty}$ is
    called then the \emph{sequence of right} (resp.\ \emph{left})
    \emph{Markov parameters of second type} of $F(z)$.
\end{enumerate}

\end{definition}
Our Definition \ref{DefMarkovParameters} is relevant to the matricial Markov parameters
introduced in~\cite[Definition 2.10]{Riv}.

\begin{remark}\label{ReConnectionSLMPandSRMP}
    Observe that~$({\bf s}_j)_{j=0}^{\infty}$ is the sequence of left Markov parameters
    of~$F(z)\in \mathbb C[z]^{p\times p}$ if and only if~$({\bf s}_j^*)_{j=0}^{\infty}$ is the
    sequence of right Markov parameters of~$F^{\vee}(z)$.
\end{remark}

If $F(z)$ is a scalar polynomial, Definitions~\ref{DefMarkovParameters}~(i) and~(ii) coincide
with the classical notion of Markov parameters. When $F(z)$ is a matrix polynomial of odd
degree, the definition of the Markov parameters of second type given in~(iii) avoids the
restriction that the even part of $F(z)$ is regular, while the first type given in~(ii) allows
to study some of polynomials with degenerate leading coefficients (which is outside the scope of
this paper). These matricial Markov parameters will play different roles in identifying Hurwitz
matrix polynomials.

Supposing that $\mathscr S\coloneqq({\bf s}_j)_{j}$ is a sequence of $p\times p$ complex
matrices, denote the finite or infinite block Hankel matrix associated with~$\mathscr S$ by
\begin{equation*}
    {H}^{\langle \mathscr S\rangle}_{j,k}\coloneqq\begin{bmatrix}
        {\bf s}_j & {\bf s}_{j+1}  & \cdots & {\bf s}_{j+k}  \\
        {\bf s}_{j+1} & {\bf s}_{j+2}  & \cdots&{\bf s}_{j+k+1} \\
        \vdots & \vdots  & \ddots & \vdots\\
        {\bf s}_{j+k} & {\bf s}_{j+k+1}  & \cdots& {\bf s}_{j+2k}
    \end{bmatrix},
\end{equation*}
where~$j \in\N_0$ and~$k\in \N_0\cup\{\infty\}$. For simplicity we
write~${H}_{k}^{\langle \mathscr S\rangle}$ for~${H}_{0,k}^{\langle \mathscr S\rangle}$.

Given $A$, $B\in \mathbb C^{p\times p}$ such that $B$ is nonsingular, denote
\begin{equation*}
\frac{A}{B}\coloneqq A\cdot B^{-1}.
\end{equation*}
Let $F(z)\in\mathbb C[z]^{p\times p}$ be a monic matrix polynomial of degree $n=2m$ or~$n=2m+1$
(whose even part~$F_e(z)$ is regular if $n=2m+1$). Let $\mathscr S$ be the associated sequence
of left or right Markov parameters (of first type when $n=2m+1$). Concerning the matricial
extension of the above stability criteria, one might pose the following questions: \emph{
    \begin{enumerate}
    \item[(Q1)] For Hurwitz stability of~$F(z)$, is it necessary or sufficient that all
        eigenvalues of the matrices~${H}^{\langle \mathscr S\rangle}_{m-1}$,
        ~${H}^{\langle \mathscr S\rangle}_{1,m-1}$ and, when nontrivial,~${\bf s}_{-1}$ are
        positive (or e.g.\ have positive real parts)?
    \item[(Q2)] For Hurwitz stability of~$F(z)$, is it necessary or sufficient that its even
        part~$F_{e}(z)$ and odd part~$F_{o}(z)$ satisfy
         \begin{equation}\label{SContinued}
             \frac{F_{o}(z)}{F_{e}(z)}
             ={\bf c}_{0}+\cfrac{1}{z{{\bf c}_1+\cfrac{1}{{\bf c}_{2m-2}
                         +\cfrac{\ddots}{z{\bf c}_{2m-1}+{\bf c}_{2m}^{-1}}}}},
    \end{equation}
    where all eigenvalues of the matrices~${\bf c}_k$ for $k=1,\ldots, n$ and, when~$n=2m+1$,
    additionally~${\bf c}_{0}$ are positive (or e.g.\ have positive real parts).
    \end{enumerate}
}
\noindent Let us show that it cannot be in general true, even when all coefficient of~$F(z)$ are assumed
to be real.

\begin{claim}
    To the necessity part of~(Q1) and~(Q2), the answer is generally negative. As a
    counterexample, consider the matrix polynomial
    \[
        F(z)\coloneqq
        \begin{bmatrix}{z^{2}+2 z+2} & {2 z+1} \\ {z+\frac{1}{2}} &
            {z^{2}+z+\frac{1}{2}}\end{bmatrix}.
    \]
    Both its right Markov parameters
    $$
    {\bf s}_0=
    \begin{bmatrix}
        2 & 2\\
        1 & 1
    \end{bmatrix}
    \quad\text{and} \quad
    {\bf s}_1=
    \begin{bmatrix}
        2 & 2\\
        1 & 1
    \end{bmatrix}
    \cdot
    \begin{bmatrix}
        2 & 1\\
        \frac 12 & \frac 12
    \end{bmatrix}
    =
    \begin{bmatrix}
        {5} & {3} \\ {\frac{5}{2}} & {\frac{3}{2}}
    \end{bmatrix}
    $$
    are singular and have one positive eigenvalue each. Furthermore, no matricial Stieltjes
    continued fraction as in \eqref{SContinued} exists. However, the zeros of~$F(z)$ may be found
    explicitely, their approximate values are
    \[
        \{-1, -1.876, -0.062 \pm 0.513 \i\};
    \]
    so, $F(z)$ is a Hurwitz matrix polynomial. Another counterexample to the necessity part
    of~(Q1) (with nondegenerate Markov parameters) may be found in
    formula~\eqref{eq:counterexample3} below.
\end{claim}

\begin{claim}
    The sufficiency in questions~(Q1) and~(Q2) does not generally hold. Indeed, if we take
    \[
        F(z) \coloneqq
        \begin{bmatrix}
            {z^{2}+z+1} & {\frac{z}{3}+\frac 12} \\[2pt] {5 z+1} & {z^{2}+2 z + 1}
        \end{bmatrix}
        ,
    \]
    the two initial right Markov parameters will be
    $$
    {\bf s}_0=
    \begin{bmatrix}
        {1} & {\frac{1}{3}} \\[2pt] {5} & {2}
    \end{bmatrix}\quad \mbox{and} \quad
    {\bf s}_1=
    \begin{bmatrix}
        {\frac{4}{3}} & {\frac{5}{6}} \\[2pt] 7 & {\frac{9}{2}}
    \end{bmatrix},
    $$
    with only positive eigenvalues. (It may be shown that eigenvalues of all Markov parameters
    in this example are positive). At the same time, there exists a matricial Stieltjes
    continued fraction of the form~\eqref{SContinued}, where
    \[
        {\bf c}_1=
        \begin{bmatrix}
            {6} & -1 \\[2pt] {-15} & {3}
        \end{bmatrix}\quad \mbox{and} \quad
        {\bf c}_2=
        \begin{bmatrix}
            {\frac{4}{3}} & -{\frac{1}{3}} \\[2pt] 6 & -1
        \end{bmatrix}.
    \]
    All eigenvalues of both matrices ${\bf c}_1$ and ${\bf c}_2$ have positive real parts.
    Nevertheless, the zeros of~$F(z)$ are approximately
    \[
    \{    -1.581 \pm 0.396 \i, 0.081 \pm 0.426 \i\},
    \]
    and hence~$F(z)$ cannot be a Hurwitz matrix polynomial.
\end{claim}

Quadratic (of degree~$2$) matrix polynomials with Hermitian coefficients may be considered as a
close analogue of real scalar polynomials. Indeed, they are necessarily stable when their
coefficients are positive definite, see e.g.~\cite[Chapter 13]{GLRMP}. However, the converse is
generally not true if the coefficient near the linear term fails to be positive definite: for
example, the real symmetric matrix polynomial
\begin{equation}\label{eq:counterexample3}
    \begin{bmatrix}
       z^2+  2z +16 & \frac{5}{2}z +1\\
        \frac{5}{2}z+1 & z^2+  2z+3
    \end{bmatrix}
\end{equation}
turns to be Hurwitz-stable. Furthermore, polynomials of higher degrees do not allow such a
direct analogy between the cases of Hermitian and scalar coefficients.

Consequently, to give a proper extension of the stability criterion via Markov parameters for
all complex/real matrix polynomials seems to be impossible. So, the following questions arise
naturally: \emph{
    \begin{enumerate}
    \item[(Q3)] Which additional conditions on the coefficients or Markov parameters of a matrix
        polynomial may be posed to make (Q1) and (Q2) true?
    \item [(Q4)] How can we formulate the stability tests via Markov parameters of second type?
    \end{enumerate}
}
Before answering them, we will find out in the next section how to calculate the number of zeros
that a matrix polynomial has in different parts of the complex plane. Then Section~\ref{Main},
will be devoted to answers to~(Q3) and~(Q4), as well as connections between Hurwitz matrix
polynomials and total positivity.

\section{Zero localization of matrix polynomials}\label{SectionRouth-HurwitzMP}

\subsection{Matricial Hermite-Fujiwara theorem revisited}

Recalling the definition of the inertia of a matrix: For~$A\in \mathbb C^{p\times p}$, the
inertia of~$A$ with respect to the imaginary axis~${\rm i}\R$ is defined by the
triple \label{PagInertia}
\begin{equation*}
\In(A)\coloneqq(\pi(A),\nu(A),\delta(A)),
\end{equation*}
where~$\pi(A)$, $\nu(A)$, and~$\delta(A)$ stand for the number of eigenvalues (counting
algebraic multiplicities) of~$A$ with positive, negative, and zero real parts, respectively.

For the inertia of regular matrix polynomials, let us adopt the notation from~\cite{LRT}, which is
essentially as in~\cite{LT82} (see also~\cite[Proposition 2.2]{LRT}).
\begin{definition}
    Let~$F(z)\in {\mathbb C}[z]^{p\times p}$ be regular. Denote by~$\gamma_+(F)$, $\gamma_-(F)$, $\gamma_0(F)$ the number of
    zeros of~$F(z)$ (counting with multiplicities), in the open upper half plane, the
    open lower half plane and on the real axis, respectively. The triple
\begin{equation*}
\gamma(F)\coloneqq(\gamma_+(F),\gamma_-(F),\gamma_0(F))
\end{equation*}
 is called \textit{the inertia
of}~$F(z)$ \textit{with respect to}~$\R$.
Analogously, the triple
\begin{equation*}
\gamma'(F)\coloneqq(\gamma'_+(F),\gamma'_-(F),\gamma'_0(F)),
\end{equation*}
is called \textit{the inertia of}~$F(z)$
\textit{with respect to}~${\rm i}\R$, replacing the upper half plane by the right half
plane, the lower half plane by the left half plane, and the real axis
by the imaginary axis, respectively.
\end{definition}

A regular matrix polynomial~$F(z)\in {\mathbb C}[z]^{p\times p}$ of degree $n$ is Hurwitz stable
if and only if~$\gamma'_-(F)=np$. So, our first step is to determine the inertia~$\gamma'(F)$ of
a matrix polynomial~$F(z)$, that is to solve the Routh-Hurwitz problem for~$F(z)$. In the scalar
case, the relation between~the inertia of $F(z)$ and the inertia of a special Hermitian matrix
--- the Bezoutian --- is well known, see e.g.~\cite{Fuj} or~\cite[pp.\,466--467]{LT85}. The
matrix case may be found in~\cite{DyV,LRT,LT82}. We recall a particular refinement of the
classical Hermite-Fujiwara theorem by Lerer and Tismenetsky~\cite{LT82}. It is stated in terms of
the generalized Bezoutian matrices and greatest common divisors of matrix polynomials.

Let matrix polynomials~$L(z),\wtilde L(z),M(z), \wtilde M(z)\in \mathbb C[z]^{p\times p}$
satisfy
\begin{equation}\label{CommonMultiple}
    \wtilde M(z)\wtilde L(z)=M(z)L(z).
\end{equation}
Then the \emph{Anderson-Jury Bezoutian matrix} ${\mathbf B}_{\wtilde M,M}(L,\wtilde L)$
\emph{associated with the quadruple}~$(\wtilde M,\wtilde L,M,L)$ is defined by:
\begin{multline}
    (I_p,z I_p,\ldots,z^{n_1-1}I_p){\mathbf B}_{\wtilde M,M}(L,\wtilde L)(I_p,u I_p,\ldots,u^{n_2-1}I_p)^{\rm T}
    \\
    =\frac1{z-u} \left(\wtilde M(z)\wtilde L(u)-M(z)L(u)\right), \label{BezoutianDefintion}
\end{multline}
where $n_1\coloneqq\max\{\deg M, \deg \wtilde M\}$ and
$n_2\coloneqq\max\{\deg L, \deg \wtilde L\}$.

The Anderson-Jury Bezoutian matrix~${\mathbf B}_{\wtilde M,M}(L,\wtilde L)$ is skew-symmetric in the
sense that
\begin{equation*}
{\mathbf B}_{\wtilde M,M}(L,\wtilde L)=-{\mathbf B}_{M,\wtilde M}(\wtilde L,L).
\end{equation*}
For commuting $L(z)$ and $\wtilde L(z)$, i.e., when~$L(z)\wtilde L(z)=\wtilde L(z)L(z)$, it is natural to choose~$\wtilde M(z)=L(z)$ and~$M(z)=\wtilde L(z)$. For a nontrivial
choice of $\wtilde M(z)$ and $M(z)$ in the general non-commutative case, we refer the reader to the
construction of the common multiples via spectral theory of matrix polynomials:
see~\cite[Theorem 9.11]{GLRMP} for the monic case and~\cite[Theorem 2.2]{GKLR} for the comonic
case.

\begin{definition}
    Suppose that $F(z), L(z)\in \mathbb C[z]^{p\times p}$. $L(z)$ is called a \emph{right}
    (resp.\ \emph{left}) \emph{divisor} of $F(z)$ if there exists an
    $M(z)\in \mathbb C[z]^{p\times p}$ such that
    \begin{equation*}
        F(z)=M(z)L(z) \quad (\mbox{resp.\ } F(z)=L(z)M(z)).
    \end{equation*}
    Let additionally $\wtilde F(z)\in \mathbb C[z]^{p\times p}$. Then~$L(z)$ is called a
    \emph{right} (resp.\ \emph{left}) \emph{common divisor} of~$F(z)$ and $\wtilde F(z)$ if
    $L(z)$ is a right (resp.\ left) divisor of $F(z)$ and also a right (resp.\ left) divisor of
    $\wtilde F(z)$.
    
    Furthermore, $L(z)$ is called a \emph{greatest right} (resp.\ \emph{left})
    \emph{common divisor} (\emph{GRCD} (resp.\ \emph{GLCD})) \emph{of $F(z)$ and} $\wtilde F(z)$
    if any other right (resp.\ left) common divisor is a right (resp.\ left) divisor of $L(z)$.

    Some basic properties of GRCDs/GLCDs are given in Appendix~\ref{sec:great-comm-divis}.
\end{definition}

\begin{lemma}[{\cite[Theorem 2.1]{LT82}}]\label{LemHermite-Fujiwara}
    Given~a regular~matrix polynomial $L(z)\in \mathbb C[z]^{p\times p}$,
    let~$L_1(z)\in \mathbb C[z]^{p\times p}$ be also regular such that
    $$L^{\vee}(z)L(z)=L^{\vee}_1(z)L_1(z). $$
    Then
\begin{align*}
\gamma_+(L)&=\pi(-{\rm i} {\mathbf B}_{L_1^{\vee},L^{\vee}}(L,L_1)) + \gamma_+(L_0),\\
\gamma_-(L)&=\nu(-{\rm i} {\mathbf B}_{L_1^{\vee},L^{\vee}}(L,L_1)) + \gamma_-(L_0),\\
\gamma_0(L)&=\delta(-{\rm i} {\mathbf B}_{L_1^{\vee},L^{\vee}}(L,L_1))-\gamma_+(L_0)-\gamma_-(L_0),
\end{align*}
where~$L_0(z)$ is a GRCD of~$L(z)$ and~$L_1(z)$.
\end{lemma}

Lemma~\ref{LemHermite-Fujiwara} describes the inertia of~$L(z)$ in terms of the inertia of the
matrix~$-{\rm i} {\mathbf B}_{L_1^{\vee},L^{\vee}}(L,L_1)$. In general, the inertia
of~$-{\rm i} {\mathbf B}_{L_1^{\vee},L^{\vee}}(L,L_1)$ depends on the choice of~$L_1(z)$. In the
scalar case, the obvious choice~$L_1(z)=L^{\vee}(z)$ leads to the classical Bezoutian which occurs in
the Hermite-Fujiwara theorem.

\subsection{Inertia representation of matrix polynomials  in terms of Hermitian Markov parameters}

Given a monic matrix polynomial~$F(z)\in \mathbb C[z]^{p\times p}$, define a related composite
polynomial~$L(z)\coloneqq F({\rm i} z)$. Then the inertia~$\gamma'(F)$ may be obtained
as~$\gamma(L)$ in Lemma \ref{LemHermite-Fujiwara}. So, the remaining question is how to choose
an appropriate~$L_1(z)$ and, accordingly, $L_0(z)$ such that the inertia of
$-{\rm i} {\mathbf B}_{L_1^{\vee},L^{\vee}}(L,L_1)$ may be expressed via the Markov parameters
of~$F(z)$. Due to this reason, we consider the case when the Markov parameters are Hermitian
matrices and apply a lemma from {\cite[Lemma 2.2]{AnJu}} bridging the Anderson-Jury Bezoutian
matrices and block Hankel matrices.

\begin{lemma}\label{LemCongruenceBezoutHankel}
    Let $N_R(z)$, $D_R(z)$, $N_L(z)$ and $D_L(z)\in \mathbb C[z]^{p\times p}$ be regular such
    that $\deg D_R>\deg N_R$, $\deg D_L>\deg N_L$ and, for all large enough~$z\in \mathbb C$,
   \[
       \big(D_L(z)\big)^{-1}N_L(z)=N_R(z)\big(D_R(z)\big)^{-1}=\sum_{k=0}^{\infty}z^{-(k+1)}{\bf s}_k.
    \]
     If~$D_L(z)$
    and~$D_R(z)$ are written in the form:
    \begin{equation*}
        D_L(z)=\sum_{k=0}^{m_L} D_{L,m_L-k} z^k
        \quad\text{and}\quad
        D_R(z)=\sum_{k=0}^{m_R} D_{R,m_R-k} z^k,
    \end{equation*}
    where~$D_{L,k}\in \mathbb C^{p\times p}$ for~$k=0,\ldots, m_L$
    and~$D_{R,k}\in \mathbb C^{p\times p}$ for~$k=0,\ldots,  m_R$, then
\begin{multline*}
  {\mathbf B}_{D_L, N_L}(D_R,N_R)
  =
       \begin{bmatrix}
           D_{L,m_L-1} & \cdots & D_{L,0}\\
           \vdots & \iddots & \\
           D_{L,0} &&
       \end{bmatrix}
  \\
  \cdot
       \begin{bmatrix}
           {\bf s}_{0} & \cdots & {\bf s}_{m_R-1}\\
           \vdots &   &  \vdots \\
           {\bf s}_{m_L-1} &  \cdots &  {\bf s}_{m_R+m_L-2}
       \end{bmatrix}
                                       \begin{bmatrix}
                                           D_{R,m_R-1} & \cdots & D_{R,0}\\
                                           \vdots & \iddots & \\
                                           D_{R,0} &&
                                       \end{bmatrix}.
\end{multline*}
\end{lemma}

\begin{theorem}\label{ThmInertiaHankelSRMPOriginal}
    Let $F(z)\in\mathbb C[z]^{p\times p}$ be monic of degree $n$. Suppose that $\mathscr S$ is
    the associated sequence of right (resp.\ left) Markov parameters for the case~$n=2m$ or
    right (resp.\ left) Markov parameters of second type for the case~$n=2m+1$. If~$\mathscr S$
    is a sequence of Hermitian matrices, then
  \begin{enumerate}
  \item[{\rm (i)}]in the even case when $n=2m$,
     \begin{align*}
      \gamma'_-(F)&=\pi(H_{m-1}^{\langle \mathscr S\rangle})+\pi({H}^{\langle \mathscr S\rangle}_{1,m-1}) +\gamma_+(\what F),\\
      \gamma'_+(F)&=\nu( H^{\langle \mathscr S\rangle}_{m-1})+\nu({H}^{\langle \mathscr S\rangle}_{1,m-1}) + \gamma_-(\what F),\\
      \gamma'_0(F)&=\delta(H_{m-1}^{\langle\mathscr S\rangle})+\delta({H}^{\langle \mathscr S\rangle}_{1,m-1})-\gamma_+(\what F)- \gamma_-(\what F),
    \end{align*}
  
\item[{\rm (ii)}]in the odd case when $n=2m+1$,
     \begin{align}
      \gamma'_-(F)&=\pi( H_{m}^{\langle\mathscr S\rangle})+\pi({H}^{\langle \mathscr S\rangle}_{1,m-1}) +\gamma_+(\what F),\label{gammaF-}\\
      \gamma'_+(F)&=\nu(H_{m}^{\langle\mathscr S\rangle})+\nu({H}^{\langle \mathscr S\rangle}_{1,m-1}) + \gamma_-(\what F),\label{gammaF+} \\
      \gamma'_0(F)&=\delta(H_{m}^{\langle\mathscr S\rangle})+\delta({H}^{\langle \mathscr S\rangle}_{1,m-1})-\gamma_+(\what F)- \gamma_-(\what F), \label{gammaF0}
    \end{align}
\end{enumerate}
where $\what F(z)$ is a GRCD (resp.\ GLCD) of~$F_{e}(-z^2)$ and~$z F_{o}(-z^2)$, and $F_e(z)$
and $F_o(z)$ are the even and odd parts of~$F(z)$, respectively.
\end{theorem}

\begin{proof} 
    We only prove the case of right Markov parameters: one can easily obtain the corresponding
    results for left Markov parameters with the help of Remarks~\ref{ReConnectionSLMPandSRMP}
    and \ref{ProGRCDGLCD}.

The proof for {\rm (i)}:
Let
\begin{align*}
\what F_{e}(z)&\coloneqq F_{e}(-z^2),\\
\what F_{o}(z)&\coloneqq z F_{o}(-z^2),\\
L(z)&\coloneqq F({\rm i} z)=\what F_{e}(z)+{\rm i}\what F_{o}(z),\\
L_1(z)&\coloneqq\what F_{e}(z)-{\rm i}\what F_{o}(z).
\end{align*}
Clear that both~$L(z)$ and~$L_1(z)$ are regular. A combination of the fact that
\begin{equation*}
\begin{bmatrix}
L(z)\\ L_1(z)
\end{bmatrix}=\begin{bmatrix} I_p & {\rm i} I_p \\ I_p & -{\rm i} I_p \end{bmatrix} \begin{bmatrix} \what F_{e}(z)\\
\what F_{o}(z)
\end{bmatrix},
 \end{equation*}
and Proposition~\ref{ProGCDUnimodular} shows that~$\what F(z)$ is also a GRCD of~$L(z)$ and~$L_1(z)$.

In the even subcase~$n=2m$, the sequence~$\mathscr S$ consists of Hermitian matrices, and hence
both~${\mathbf B}_{L_1^{\vee},L^{\vee}}(L,L_1)$ and
${\mathbf B}_{\what F_{e}^{\vee},-\what F_{o}^{\vee}}(\what F_{e},-\what F_{o})$ are
well-defined. Moreover, \begin{equation}\label{EQBezoutianCombination}
    {\mathbf B}_{L_1^{\vee},L^{\vee}}(L,L_1) =2{\rm i} {\mathbf
        B}_{\what F_{e}^{\vee},-\what F_{o}^{\vee}}(\what F_{e},-\what F_{o}).
\end{equation}
As a result,
\begin{align}
\gamma'_-(F)
  &=\gamma_+(L)=\pi\big(-{\rm i} {\mathbf B}_{L_1^{\vee},L^{\vee}}(L,L_1)\big)
    +\gamma_+(\what F)\nonumber \\
  &=\pi\left(2{\mathbf B}_{\what F_{e}^{\vee},-\what F_{o}^{\vee}}(\what F_{e},-\what F_{o})\right)
    +\gamma_+(\what F) \label{Gamma-F}
\end{align}
where the 2nd equality is due to Lemma~\ref{LemHermite-Fujiwara} and the last is due
to~\eqref{EQBezoutianCombination}.

Assume that $\mathscr S\coloneqq({\bf s}_k)_{k=0}^{\infty}$.
Lemma~\ref{LemCongruenceBezoutHankel} shows that
${\mathbf B}_{\what F_{e}^{\vee},-\what F_{o}^{\vee}}(\what F_{e},-\what F_{o})$ is congruent to
the block Hankel matrix
\begin{equation*}
  \begin{bmatrix} {\bf s}_0 & 0_p & {\bf s}_1 & \cdots & {\bf s}_{m-1} & 0_p\\
        0_p & {\bf s}_1 & 0_p & \cdots &0_p & {\bf s}_{m} \\
        {\bf s}_1 & 0_p & {\bf s}_2 & \cdots & {\bf s}_m & 0_p\\
        \vdots & \vdots & \vdots && \vdots & \vdots\\
        {\bf s}_{m-1} & 0_p & {\bf s}_m & \cdots & {\bf s}_{2m-2} & 0_p\\
        0_p & {\bf s}_m & 0_p & \cdots & 0_p & {\bf s}_{2m-1}
\end{bmatrix}.
\end{equation*}
By reordering rows and columns, it is subsequently congruent to
\begin{equation*}
 \begin{bmatrix}
 H_{m-1}^{\langle\mathscr S\rangle} & \\
 & {H}^{\langle \mathscr S\rangle}_{1,m-1}
 \end{bmatrix}.
\end{equation*}
Hence, \eqref{Gamma-F} implies that
\begin{equation*}
 \gamma'_-(F)=\pi\big(H_{m-1}^{\langle\mathscr S\rangle}\big)
    +\pi\big({H}^{\langle \mathscr S\rangle}_{1,m-1}\big)
    +\gamma_+(\what F),
\end{equation*}
Analogously, we deduce
\begin{align*}
\gamma'_+(F)&=\nu(H_{m-1}^{\langle\mathscr S\rangle})+\nu({H}^{\langle \mathscr S\rangle}_{1,m-1}) + \gamma_-(\what F),\\
\gamma'_0(F)&=\delta(H_{m-1}^{\langle\mathscr S\rangle})+\delta({H}^{\langle \mathscr S\rangle}_{1,m-1})-\gamma_+(\what F)- \gamma_-(\what F),
\end{align*}
Therefore, Theorem~\ref{ThmInertiaHankelSRMPOriginal} is verified for the even case {\rm (i)}.

The odd case {\rm (ii)} follows in an analogous way. 
\end{proof}

\section{Main results}\label{Main}
The following three subsections are devoted to the identification of Hurwitz matrix polynomials
based on the inertia representation from Theorem \ref{ThmInertiaHankelSRMPOriginal}.

\subsection{Hurwitz stability and Stieltjes positive definiteness}
Let~$F(z)\in \mathbb C[z]^{p\times p}$ be monic of degree $n$ with the even part $F_e(z)$ and
odd part $F_o(z)$. Assume that $\mathscr S\coloneqq ({\bf s}_k)_{k=0}^{\infty}$ be the sequence
of right or left Markov parameters of $F(z)$ (of the second type for~$n=2m+1$). Then there
exists a one-to-one correspondence between the coefficients of $F(z)$, ~$\mathscr S$ and the
truncated sequence $({\bf s}_k)_{k=0}^{n-1}$. For simplicity, we only exhibit this correspondence
for the case when~$n=2m$ and~$\mathscr S$ is the sequence of right Markov parameters. Comparing
both sides of \eqref{MakrovPE} shows that
\begin{equation}\label{HC}
{H}^{\langle \mathscr S\rangle}_{j+k,m-1}={H}^{\langle \mathscr S\rangle}_{j,m-1}C^k,\quad j,k\in \N_0,
\end{equation}
and 
\begin{equation}\label{ASA}
\begin{bmatrix}
A_{2m-1}\\
A_{2m-3}\\
\vdots\\
A_1
\end{bmatrix}=\begin{bmatrix}
{\bf s}_0 & -{\bf s}_1 & \cdots & (-1)^{m-1}{\bf s}_{m-1}\\
 & {\bf s}_0  & \cdots & (-1)^{m-2}{\bf s}_{m-2}\\
  & & \ddots & \vdots\\
  & & & {\bf s}_0
\end{bmatrix}\begin{bmatrix}
A_{2m-2}\\
A_{2m-4}\\
\vdots\\
A_0
\end{bmatrix}
\end{equation}
where $F(z)$ is written as in~\eqref{ReprentationMatrixPolynomials} and
\begin{equation}\label{Cmatrix}
         {C}\coloneqq\begin{bmatrix}
            0_p & \cdots & 0_p &-A_{2m} \\
            I_p & \cdots & 0_p & A_{2m-2}\\
            \vdots & \ddots & \vdots & \vdots\\
            0_p & \cdots & I_p & (-1)^m A_{2}\\
        \end{bmatrix}
\end{equation}
is the so-called companion matrix of $F_e(-z)$.  The correspondence is accordingly obtained. 

What is more, it deserves special attention if the truncated sequence of Markov parameters
\begin{equation}\label{skH}
    ({\bf s}_k)_{k=0}^{n-1} \mbox{ is  a sequence of Hermitian matrices}.
\end{equation}
If so, ${\bf s}_k$ is also Hermitian for $k=n,n+1,\ldots$. Indeed, without loss of generality we
again consider the case of $n=2m$ and $\mathscr S$ being the sequence of right Markov
parameters. As an immediate consequence of \eqref{HC}, we have
\begin{equation}
    {H}^{\langle \mathscr
        S\rangle}_{2k+l,m-1}=\left(C^k\right)^*{H}_{l,m-1}^{\langle
        \mathscr S\rangle}C^k,\quad\text{for all}\quad k,l\in \N_0
    . \label{H2kCHC}
\end{equation}
The fact on~${\bf s}_k$ follows since the left-hand side of~\eqref{H2kCHC} is Hermitian. So,
these Hermitian matrices ${\bf s}_k$ are a natural extension of classical real Markov
parameters.

It is important for us that the condition~\eqref{skH} allows associating this
sequence $({\bf s}_k)_{k=0}^{n-1}$ with a truncated matricial Stieltjes moment problem and to
see more connection when $F(z)$ is a Hurwitz matrix polynomial.

Given a sequence of Hermitian matrices~$({\bf s}_k)_{k=0}^{n-1}$, truncated matricial Stieltjes
moment problem of first type (resp.\ second type) is to find all the nonnegative Hermitian $p\times p$
Borel measures $\tau$ on $[0,\infty)$ such that
\begin{multline*}
{\bf s}_i=\int_{[0,\infty)} u^{i} {\rm d} \tau( u),\quad i=0,\ldots, n-1\\
\Big(
\mbox{resp}.\quad {\bf s}_i=\int_{[0,\infty)} u^{i} {\rm d} \tau( u), \quad i=0,\ldots, n-2,\quad {\bf s}_{n-1}-\int_{[0,\infty)} u^{n-1} {\rm d} \tau( u)\succeq 0\Big).
\end{multline*}

For the detailed study of these matrix moment problems, we refer the reader to
\cite{AdTk1,AdTk2,An70,Bol,CHStieltjes,DFKM,HC}. Recall that the well-known solvability criteria for these
problems (see e.g., Lemma 1.7 of \cite{Bol}, Lemma 1.2 of \cite{CHStieltjes}, Theorem 1.4 of
\cite{DFKM}, Theorem 1.1 of \cite{HC}) are related to some special matrix sequences as follows:

\begin{definition}\label{DefHankel}
 Let $l\in \N_0$ and let $({\bf s}_k)_{k=0}^l$ be a finite  sequence of $p\times p$ Hermitian matrices.
 \begin{enumerate}
 \item[{\rm (i)}]  $({\bf s}_k)_{k=0}^l$ is called \textit{Stieltjes positive definite} (resp.\ \textit{Stieltjes nonnegative definite}) if both block Hankel matrices $[{\bf s}_{j+k}]_{j,k=0}^{\lfloor\frac{l}2\rfloor}$ and $[{\bf s}_{j+k+1}]_{j,k=0}^{\lfloor\frac{l-1}2\rfloor}$  are positive definite (resp.\ nonnegative definite).  
 \item[{\rm (ii)}] $({\bf s}_j)_{j=0}^l$ is called \textit{Stieltjes nonnegative definite extendable} if there exists a complex $p\times p$ matrix ${\bf s}_{l+1}$  such that $({\bf s}_j)^{l+1}_{j=0}$ is Stieltjes nonnegative definite.
 \end{enumerate}
 \end{definition}

\begin{theorem}\label{ThmStieltjesMomentProblemEqual}
    Let $l\in \N_0$ and let $({\bf s}_k)_{k=0}^l$ be a finite sequence of $p\times p$ Hermitian
    matrices. Then~$({\bf s}_k)_{k=0}^l$ corresponds to a solvable truncated matricial Stieltjes
    moment problem of first type if and only if $({\bf s}_k)_{k=0}^l$ is Stieltjes nonnegative
    definite extendable.
\end{theorem}

\begin{theorem}\label{ThmStieltjesMomentProblemLeq}
    Let $l\in \N_0$ and let $({\bf s}_k)_{k=0}^l$ be a finite sequence of $p\times p$ Hermitian
    matrices. Then~$({\bf s}_k)_{k=0}^l$ corresponds to a solvable truncated matricial Stieltjes
    moment problem of second type if and only if $({\bf s}_k)_{k=0}^l$ is Stieltjes nonnegative
    definite.
\end{theorem}

Given a Stieltjes positive definite sequence $({\bf s}_k)_{k=0}^l$, the related truncated
matricial Stieltjes moment problem is solvable with an infinite number of solutions. Via the
Dyukarev–Stieltjes parameters of $({\bf s}_k)_{k=0}^l$, the so-called Dyukarev matrix
polynomials form the resolvent matrix for this problem (see \cite[Theorem 7]{Dyu}). In what
follows, we link Stieltjes positive definite sequences with Hurwitz stability of matrix
polynomials.

\begin{theorem}\label{ThmHurwitzMatrixHurwitzSRMP}
    Let $F(z)\in\mathbb C[z]^{p\times p}$ be monic of degree $n$. Assume that $\mathscr S$ is
    the related truncated Hermitian sequence of right Markov parameters (of second type when
    $n=2m+1$). Then $F(z)$ is a Hurwitz matrix polynomial if and only if $\mathscr S$ is
    Stieltjes positive definite.
\end{theorem}

\begin{proof}
 
    The proof for ``if'' implication: Assume that $ \mathscr S$ is a Stieltjes positive definite
    matrix sequence. Accordingly, Lemma \ref{ThmInertiaHankelSRMPOriginal} yields that
\begin{equation*}
np=\deg\det F(z)\geq \gamma'_-(F)\geq \pi(H_{m-1}^{\langle \mathscr S\rangle})+ \pi({H}^{\langle \mathscr S\rangle}_{1,m-1})=np,
\end{equation*}
for the even case and
\begin{equation*}
np=\deg\det F(z)\geq \gamma'_-(F)\geq \pi(H_{m}^{\langle \mathscr S\rangle})+ \pi({H}^{\langle \mathscr S\rangle}_{1,m-1})=np
\end{equation*}
for the odd case. Thus $\gamma'_-(F)=np$, which means that $F(z)$ is a Hurwitz matrix
polynomial.

The proof for ``only if'' implication: Let $F(z)$ be a Hurwitz matrix polynomial with the
sequence of right Markov parameters $\mathscr S$. Then $\gamma'_{-}(F)=np$. Denote by
$\what F(z)$ a GRCD of $F_{e}(-z^2)$ and $z F_{o}(-z^2)$, where $F_{e}(z)$ and $F_o(z)$ are the
even part and odd part of $F(z)$, respectively. Suppose that~$\sigma(\what F)\ne\varnothing$.
One can see from Proposition \ref{ProGCDUnimodular} that $\what F(z)$ is a GRCD of
\[
    F({\rm i} z)=F_{e}(-z^2) +{\rm i} z F_{o}(-z^2)
    \quad\text{and}\quad F(-{\rm i} z)=F_{e}(-z^2) +{\rm i} z F_{o}(-z^2).
\]
Then for each zero $z_0\in\sigma(\what F)$, we have~${\rm i} z_0\in \sigma(F)$
and~$-{\rm i} z_0\in \sigma(F)$, which contradicts the assumption that~$F(z)$ is a Hurwitz
matrix polynomial. Consequently, Lemma~\ref{ThmInertiaHankelSRMPOriginal} yields the expressions
\begin{equation*}
np= \pi(H_{m-1}^{\langle\mathscr S\rangle})+ \pi({H}^{\langle \mathscr S\rangle}_{1,m-1})
\end{equation*}
for the even case and
\begin{equation*}
np= \pi(H_{m}^{\langle \mathscr S\rangle})+ \pi({H}^{\langle \mathscr S\rangle}_{1,m-1})
\end{equation*}
for the odd case,
from which the statement (ii) follows.

\end{proof}

\begin{remark}\label{RemarkHurwitzSLMP}
    As follows from Remark \ref{ReConnectionSLMPandSRMP}, if the right Markov parameters in
    Theorem~\ref{ThmHurwitzMatrixHurwitzSRMP} are substituted by the left Markov parameters, the
    corresponding results hold true as well.
\end{remark}

We provide a few examples of testing the Hurwitz stability for matrix polynomials via Theorem
\ref{ThmHurwitzMatrixHurwitzSRMP}.

\begin{example} Let $F(z)\in \mathbb C[z]^{2\times 2}$ of degree $3$ be given as
\begin{multline*}
F(z)\coloneqq \begin{bmatrix}1 & 0\\ 0 &1\end{bmatrix}z^3+\begin{bmatrix}3 & 4\\ 4 &8\end{bmatrix}z^2+\begin{bmatrix}
23-15{\rm i} & 33+35{\rm i}\\ 12-10{\rm i} & 17+15{\rm i}\end{bmatrix}z\\
+\begin{bmatrix} 115-85{\rm i} & 170+165{\rm i}\\ 191-140{\rm i} & 261+260 {\rm i}
\end{bmatrix}.
\end{multline*}
The related right Markov parameters are
\begin{equation*}
{\bf s}_{0}=\begin{bmatrix}3 & 4\\ 4 &8\end{bmatrix},\quad {\bf s}_1=\begin{bmatrix}2 & -3 \\-3 &7\end{bmatrix},\quad {\bf s}_2=\begin{bmatrix}10 & 15+25 {\rm i}\\ 15-25{\rm i} &20\end{bmatrix}.
\end{equation*}
They satisfy ${\bf s}_{0}\succ 0$ and ${\bf s}_1\succ 0$, however
\begin{equation*}
{\bf s}_2-{\bf s}_1{\bf s}_0^{-1} {\bf s}_1=\begin{bmatrix}-\frac{27}{8} & \frac{323}{8}+25{\rm i} \\\frac{323}{8}-25{\rm i} & -\frac{227}{8}\end{bmatrix}
\end{equation*}
is not positive definite. So $F(z)$ cannot be a Hurwitz matrix polynomial.
\end{example}

\begin{example}
Let $F(z)\in \mathbb C[z]^{3\times 3}$ of degree $3$ be given as
\begin{multline*}
    F(z)\coloneqq
    \begin{bmatrix}1 & 0 & 0\\ 0 &1 & 0 \\ 0 & 0 & 1\end{bmatrix}z^3
    +
    \begin{bmatrix}2 & -1-{\rm i} & {\rm i}\\-1+ {\rm i} &2 & -1\\ -{\rm i} & -1 & 2 \end{bmatrix}z^2\\
    +
    \begin{bmatrix}
        65+3{\rm i} & -6 & 70{\rm i} \\
        -1+30{\rm i} & 5-3{\rm i} & -35-3{\rm i} \\
        2-20{\rm i} & 4{\rm i} & 40 
    \end{bmatrix} z
    +
    \begin{bmatrix}
        180-21{\rm i} & -24-3{\rm i} & 32+219{\rm i} \\
        -72+143{\rm i} & 14-16{\rm i} & -180-76{\rm i} \\
        8-136{\rm i} & -5+17{\rm i} & 182+3{\rm i}
    \end{bmatrix}.
\end{multline*}
The related right Markov parameters are
\begin{align*}
&{\bf s}_{0}=\begin{bmatrix}2 & -1-{\rm i} & {\rm i}\\-1+ {\rm i} &2 & -1\\ -{\rm i} & -1 & 2 \end{bmatrix},
\qquad {\bf s}_1=\begin{bmatrix}1 & {\rm i} & -{\rm i}\\- {\rm i} &2 & 0\\ {\rm i} & 0 & 3 \end{bmatrix},\\
& {\bf s}_2=\begin{bmatrix}15 & 1+5 {\rm i} & 3-5{\rm i}\\ 1-5{\rm i} & 10 & -6{\rm i}\\ 3+5{\rm
        i} &6{\rm i} & 50\end{bmatrix}
                       .
\end{align*}
They satisfy ${\bf s}_{0}\succ 0$, ~${\bf s}_1\succ 0$ and \begin{equation*}
{\bf s}_2-{\bf s}_1{\bf s}_0^{-1} {\bf s}_1=\begin{bmatrix}10 &-\frac{3}{2}-\frac{1}{2}{\rm i} & -\frac{1}{2}+\frac{1}{2}{\rm i} \\[2pt] -\frac{3}{2}+\frac{1}{2}{\rm i} & \frac{1}{2} & 1+\frac{1}{2}{\rm i} \\[2pt]
-\frac{1}{2}-\frac{1}{2}{\rm i} & 1-\frac{1}{2}{\rm i} & \frac{73}{2}\end{bmatrix} \succ 0.
\end{equation*}
So $F(z)$ is a Hurwitz matrix polynomial.
\end{example}

\begin{example} Let $F(z)\in \mathbb C[z]^{2\times 2}$ of degree $4$ be given as
\begin{align*}
    F(z)\coloneqq &
    \begin{bmatrix}1 & 0\\ 0 &1
    \end{bmatrix}z^4
    +
    \begin{bmatrix}
        2 & 2-{\rm i}\\
        2+{\rm i} &3
    \end{bmatrix}z^3
    +
    \begin{bmatrix}
        -58 & 5+39{\rm i}\\
        9-71{\rm i} & -67
    \end{bmatrix}z^2
    \\
    &+
    \begin{bmatrix}
        -143+83{\rm i} & -100+176{\rm i} \\
        -115-210{\rm i} & -251-151{\rm i}
    \end{bmatrix}z
    +
    \begin{bmatrix}
        23+{\rm i} & -2-17{\rm i}\\
        1+39{\rm i} & 20-5{\rm i}
    \end{bmatrix}
    .
\end{align*}
The related left Markov parameters are
\begin{multline*}
\mathscr S=\left(\begin{bmatrix} 2& 2-{\rm i}\\
2+{\rm i} & 3\end{bmatrix},\
\begin{bmatrix} -2 & -1-{\rm i} \\
-1+{\rm i} &-3\end{bmatrix},\right.
\\
\left. \begin{bmatrix} 13 & 2+13{\rm i} \\
 2-13{\rm i} &20\end{bmatrix}, \
  \begin{bmatrix} -210 & -{\rm i} \\
  {\rm i} & -377\end{bmatrix},\cdots\right).
\end{multline*}
It turns out that ${H}_{1}^{\langle \mathscr S\rangle}$ is positive definite, but
${H}_{1,1}^{\langle \mathscr S\rangle}$ is not positive definite. So $F(z)$ is not a Hurwitz matrix
polynomial.
\end{example}

\subsection{Matricial versions of  the stability criteria}

Theorem \ref{ThmHurwitzMatrixHurwitzSRMP} for scalar polynomials of even degrees coincides with
the stability criterion via Markov parameters. Now, suppose
that~$F(z)\in\mathbb C[z]^{p\times p}$ is a monic matrix polynomial of odd degree~$2m+1$ whose
even part is regular, and whose sequence of right Markov parameters of first type
is~$({\bf s}_k)_{k=-1}^{\infty}$. Analogous to the even case, we pay special attention to the
case that the truncated sequence of Markov parameters
\begin{equation}\label{sk-12m-1}
    ({\bf s}_k)_{k=-1}^{2m-1} \mbox{ is  a sequence of Hermitian matrices}.
\end{equation}
 We formulate the matrix analogue of the stability
criterion via Markov parameters in a unified way covering both even and odd cases.

\begin{theorem}\label{ThmHurwitzMatrixHurwitzOdd}
    Suppose that $F(z)\in\mathbb C[z]^{p\times p}$ is a monic matrix polynomial of degree~$n=2m$
    or $n=2m+1$, and that~$F_e(z)$ is regular for $n=2m+1$. Let~$\mathscr S$ be the related
    truncated Hermitian sequence of right Markov parameters (of first type when $n=2m+1$).
    Then~$F(z)$ is Hurwitz-stable if and only
    if~${H}^{\langle {\mathscr S}\rangle}_{m-1},\,
    {H}^{\langle {\mathscr S}\rangle}_{1,m-1}\succ 0$ and, for~$n=2m+1$, additionally
    ${\bf s}_{-1}\succ 0$, where ${\bf s}_{-1}$ is the first element of $\mathscr S$.
\end{theorem}

\begin{proof} 
    The even case~$\deg F=2m$ reduces to Theorem~\ref{ThmHurwitzMatrixHurwitzSRMP}.\\
    The odd case~$\deg F=2m+1$: put $\mathscr S:=({\bf s}_k)_{k=-1}^{\infty}$, write the even
    and the odd parts of~$F(z)$ as
    \[
        F_e(z)\coloneqq \sum_{k=0}^m z^{m-k} E_k
        \quad\text{and}\quad
        F_o(z)\coloneqq \sum_{k=0}^m z^{m-k} O_k,\quad O_0=I_p,
    \]
    and put
    \[
        F_{-1}(z)\coloneqq \left({\bf s}_0+z{\bf s}_{-1}\right)F_e(z)-zF_o(z), \quad
        F_{o,z}(z)\coloneqq zF_o(z).
    \]
    Since~$z F_o(z)=z\,{\bf s}_{-1}F_e(z) + {\bf s}_0F_e(z) + O(z^{m-1})$, we
    have~$\deg F_{-1}=m-1$. If~$\wtilde {\mathscr S}$ is the sequence of right Markov parameters
    of second type associated with~$F(z)$, then
\begin{align*}
&\begin{bmatrix}I_p\\ \vdots\\z^{m}I_p\end{bmatrix}^{\rm T} \begin{bmatrix}
E_{m}^* & \cdots & E_{0}^*\\
\vdots & \iddots & \\
E_{0}^* &&
\end{bmatrix}\begin{bmatrix}{\bf s}_{-1} & \\ & {H}^{\langle{\mathscr S}\rangle}_{1,m-1} \end{bmatrix}\begin{bmatrix}
E_{m} & \cdots & E_{0}\\
\vdots & \iddots & \\
E_{0} &&
\end{bmatrix}\begin{bmatrix}I_p\\ \vdots\\ u^{m}I_p\end{bmatrix}\\
  &=\left(I_p,\ldots,z^{m-1}I_p\right)
    {\mathbf B}_{F_{-1}^{\vee},F_e^{\vee}}(F_{-1},F_e)
    \begin{bmatrix}
        I_p\\ \vdots\\ u^{m-1}I_p
    \end{bmatrix}
  +F_e^{\vee}(z){\bf s}_{-1} F_e(u)\\
&=\frac{1}{z-u}\left(F_e^{\vee}(z)F_{-1}(u)-F_{-1}^{\vee}(z)F_e(u)\right)+F_e^{\vee}(z){\bf s}_{-1} F_e(u)\\
&=\frac{1}{z-u}\left(zF_o^{\vee}(z)F_e(u)-F_e^{\vee}(z)uF_o(u)\right)\\
&=\left(I_p,\ldots,z^{m}I_p\right){\mathbf B}_{F_e^{\vee},F_{o,z}^{\vee}}(F_e,F_{o,z})(I_p,\ldots,u^{m}I_p)^{\rm T}\\
&=\begin{bmatrix}I_p\\ \vdots\\z^{m}I_p\end{bmatrix}^{\rm T} \begin{bmatrix}
O_{m}^* & \cdots & O_{0}^*\\
\vdots & \iddots & \\
O_{0}^* & &
\end{bmatrix} H^{\langle \wtilde {\mathscr S}\rangle}_{m}\begin{bmatrix}
O_{m} & \cdots & O_{0}\\
\vdots & \iddots & \\
O_{0} & &
\end{bmatrix} 
\begin{bmatrix}I_p\\ \vdots\\ u^{m}I_p
\end{bmatrix},
\end{align*}
where the $1$-st and the last equations are due to Lemma \ref{LemCongruenceBezoutHankel}. It
follows that~$H_m^{\langle \wtilde {\mathscr S}\rangle}$ is Hermitian iff both
$H^{\langle \mathscr S \rangle}_{1,m-1}$ and $ {\bf s}_{-1}$ are Hermitian. Moreover,
$H_m^{\langle \wtilde {\mathscr S}\rangle}\succ 0$ is equivalent to
$H^{\langle \mathscr S \rangle}_{1,m-1},\,{\bf s}_{-1}\succ 0$.

Analogously, one can prove that ${H}^{\langle {\mathscr S}\rangle}_{m-1}$ is Hermitian (resp.\
positive definite) iff $H_{1,m-1}^{\langle \wtilde {\mathscr S}\rangle}$ is Hermitian (resp.\
positive definite). In view of Theorem \ref{ThmHurwitzMatrixHurwitzSRMP}, we complete the proof.
\end{proof}

\begin{remark}
    In Theorem \ref{ThmHurwitzMatrixHurwitzOdd}, the right Markov parameters may be substituted
    by the left Markov parameters: the corresponding result holds true, as is seen from
    Remark~\ref{ReConnectionSLMPandSRMP}.
\end{remark}

The proof of Theorem~\ref{ThmHurwitzMatrixHurwitzOdd} shows that, given a stable monic matrix
polynomial~$F(z)$ of odd degree, its truncated sequence of right Markov parameters of second
type satisfies the condition~\eqref{skH} if and only if the even part of~$F(z)$ is regular (so
that~$F(z)$ has well-defined right Markov parameters of first type) and~\eqref{sk-12m-1} holds.
In other words: a matrix polynomial whose stability can be tested via
Theorem~\ref{ThmHurwitzMatrixHurwitzSRMP}, but not via Theorem~\ref{ThmHurwitzMatrixHurwitzOdd},
cannot be Hurwitz stable.

Theorem \ref{ThmHurwitzMatrixHurwitzSRMP} indeed provides a situation where the answer to
question~(Q1) from Section~\ref{SectionHurwitzStieltjes} is positive and may be written in a
simple form. Based on Theorem~\ref{ThmHurwitzMatrixHurwitzSRMP} and \cite[Theorem 7.10]{Riv}, we
can now deduce the following connection between Hurwitz matrix polynomials and matricial
Stieltjes continued fractions and thereby complete our answer to question~(Q3).

\begin{theorem}\label{ThmHurwitzMatrixHurwitzParameterEven}
Let  $F(z)\in\mathbb C[z]^{p\times p}$ be monic of degree $n$ with  the even part $F_e(z)$  and
odd part $F_o(z)$. Assume that~$\mathscr S$  is the related truncated Hermitian sequence of
right Markov parameters (of second type when $n=2m+1$). Then
\begin{enumerate}
\item[{\rm (i)}] in the even case $n=2m$,
 $F(z)$ is a Hurwitz matrix polynomial if and only if
there exists a  sequence of $p\times p$ positive definite matrices $({\bf c}_k)_{k=1}^{n}$ such that the identity
\begin{equation}\label{S-fractionEven}
\frac{F_{o}(z)}{F_{e}(z)}=\cfrac{I_p}{z{\bf c}_1+\cfrac{I_p}{{\bf c}_2+\cfrac{I_p}{{\bf c}_{n-2}+\cfrac{\ddots}{z{\bf c}_{n-1}+{\bf c}_{n}^{-1}}}}}
\end{equation}
holds for large enough $z\in \mathbb C$.
\item[{\rm (ii)}] in the odd case $n=2m+1$, the following statements are equivalent:
\begin{enumerate}
\item $F(z)$ is a Hurwitz matrix polynomial.
\item  There exists a  sequence of $p\times p$ positive definite matrices $({\bf c}_k)_{k=1}^{n}$ such that the identity
\begin{equation}\label{S-fractionOdd}
\frac{F_{e}(z)}{ F_{o}(z)}=\cfrac{I_p}{{\bf c}_{1}+\cfrac{I_p}{z{\bf c}_2+\cfrac{I_p}{{\bf c}_{n-2}+\cfrac{\ddots}{z{\bf c}_{n-1}+{\bf c}_{n}^{-1}}}}}
\end{equation}
holds for large enough $z\in \mathbb C$.
\item $F_{e}(z)$ is regular and 
there exists a  sequence of $p\times p$ positive definite matrices $({\bf c}_k)_{k=1}^{n}$ such that the identity
\begin{equation*}
\frac{F_{o}(z)}{ F_{e}(z)}={\bf c}_{1}+\cfrac{I_p}{z{\bf c}_2+\cfrac{I_p}{{\bf c}_{n-2}+\cfrac{\ddots}{z {\bf c}_{n-1}+ {\bf c}_{n}^{-1}}}}
\end{equation*}
holds for large enough $z\in \mathbb C$.
\end{enumerate}
\end{enumerate}
 \end{theorem}

 In \cite{Riv}, matrix polynomials satisfying \eqref{S-fractionEven} or \eqref{S-fractionOdd}
 are called matrix Hurwitz type polynomials. In this sense, the notions ``matrix Hurwitz type
 polynomial'' and ``Hurwitz matrix polynomial'' are equivalent. In the scalar case when $p=1$,
 the statement (i) and the equivalence between (a) and (c) of (ii) are indeed the classical
 stability criterion via continued fractions.
 
\subsection{Hurwitz matrix polynomials, block-Hankel minors and  quasiminors}
The rest of this paper is devoted to block-Hankel minors and block-Hankel quasiminors in connection to
Hurwitz matrix polynomials.

We begin with an introduction of quasideterminants, which play an important role in
noncommutative algebra as determinants do in commutative algebra. The most general
definitions~\cite[(1.1), P. 92]{GR91} and~\cite[Definition 1.2.5]{GGRW} become simpler in the
particular case we need --- for quasideterminants of block matrices over~$\C$.
\begin{definition}\label{DefQuasi}
Let $l\in \N$ and $l\leq 2$ and let ${\bf M}\in \mathbb C^{pl\times pl}$ with a block decomposition ${\bf M}\coloneqq [M_{jk}]_{j,k=1}^l$, where $M_{jk}\in \mathbb C^{p\times p}$.
 Suppose that ${\bf M}_{(l;l)}\coloneqq [M_{jk}]_{j,k=1}^{l-1}\in \mathbb C^{(l-1)p\times (l-1)p}$ is nonsingular. The \emph{quasideterminant of} ${\bf M}$ \emph{with index} $(l,l)$, denoted by  $\left|{\bf M} \right|_{ll}$, is the following expression
 \begin{equation*}
\left| {\bf M} \right|_{ll}\coloneqq M_{ll}- \left[M_{l1}, \cdots, M_{l,l-1}\right]  {\bf M}_{(l;l)}^{-1}
\begin{bmatrix}
M_{1l}\\
\vdots\\
M_{l-1,l}
\end{bmatrix}.
\end{equation*}
\end{definition}
In fact, $\left| {\bf M} \right|_{ll}$ in our setting coincides with the Schur complement
of~${\bf M}_{(l;l)}$ in ${\bf M}$.

\begin{definition}
Let 
\begin{equation}\label{infinitematrix}
H\coloneqq [{\bf s}_{j+k}]_{j,k=0}^{\infty}
\end{equation}
 be an infinite block Hankel matrix 
 with ${\bf s}_{k}\in \mathbb C^{p\times p}$. 
 For $l\in \N$ and $l\geq 2$, let
\begin{equation}\label{tildeHl}
\wtilde H_l\coloneqq \begin{bmatrix}
{\bf s}_{j_1+k_1} & {\bf s}_{j_1+k_2} & \cdots & {\bf s}_{j_1+k_l}\\
{\bf s}_{j_2+k_1} & {\bf s}_{j_2+k_2} & \cdots & {\bf s}_{j_2+k_l}\\
\vdots & \vdots && \vdots\\
{\bf s}_{j_{l}+k_1} & {\bf s}_{j_{l}+k_2} & \cdots & {\bf s}_{j_{l}+k_l}
\end{bmatrix}\in \mathbb C^{pl\times pl}
\end{equation}
be  a submatrix  of $H$, where $0\leq k_1< k_2<\cdots < k_l$ and $0\leq j_1<j_2<\cdots < j_{l}$. Then
\begin{enumerate}
\item[{\rm (i)}] $\det\wtilde H_l$ is called  a block-Hankel minor of order $l$ of $H$.
\item[{\rm (ii)}] if $|\wtilde H_l|_{ll}$ is well-defined,  $|\wtilde H_l|_{ll}$ is called  a block-Hankel quasiminor of order~$l$ of~$H$.
\end{enumerate}
 Moreover, if $j_l-j_{l-1}=\cdots=j_2-j_1=1$ and $k_l-k_{l-1}=\cdots=k_2-k_1=1$, then \begin{enumerate}
\item[{\rm (i)}] $\det\wtilde H_l$ is called  a contiguous block-Hankel minor of order $l$ of $H$.
\item[{\rm (ii)}] if $|\wtilde H_l|_{ll}$ is well-defined, $|\wtilde H_l|_{ll}$ is called a
    contiguous block-Hankel quasiminor of order $l$ of $H$.
\end{enumerate} 
\end{definition}

For any given matrix polynomial, we can describe the vanishing behaviour for large block-Hankel
minors built from matricial Markov parameters.

\begin{proposition}\label{VanishHankelMinors}
    Let $F(z)\in\mathbb C[z]^{p\times p}$ be monic of degree $n=2m$ or $n=2m+1$. Assume that
    $\mathscr S$ is the related sequence of right Markov parameters (of first type for~$n=2m+1$
    if the even part~$F_e(z)$ is regular). Then all block-Hankel minors
    of~$H^{\langle \mathscr S\rangle}_{\infty}$ of order $>m$ are equal to zero.
\end{proposition}

\begin{proof}
We only give a proof for $n=2m$. Let  $\wtilde H_l$ be any given $lp\times lp$ submatrix  of~$H$ as in \eqref{tildeHl}. 
Comparing  both sides of \eqref{MakrovPE} yields that
\begin{equation}\label{SSC}
\begin{bmatrix}
{\bf s}_{t+k} &\cdots & {\bf s}_{t+k+m-1}\\
{\bf s}_{t+k+1} &\cdots & {\bf s}_{t+k+m}\\
\vdots & & \vdots\\
{\bf s}_{t+k+j_l} &\cdots & {\bf s}_{t+k+j_l+m-1}
\end{bmatrix}=\begin{bmatrix}
{\bf s}_{k} &\cdots & {\bf s}_{k+m-1}\\
{\bf s}_{k+1} &\cdots & {\bf s}_{k+m}\\
\vdots & & \vdots\\
{\bf s}_{k+j_l} &\cdots & {\bf s}_{k+j_l+m-1}
\end{bmatrix}C^t,
\end{equation}
where $C$ is as in \eqref{Cmatrix}.  Using \eqref{SSC}, we have
\begin{equation*}
\begin{bmatrix}
{\bf s}_{k_1} & \cdots & {\bf s}_{k_l} \\
{\bf s}_{k_1+1} & \cdots & {\bf s}_{k_l+1}\\
\vdots & & \vdots\\
{\bf s}_{k_1+j_l}& \cdots & {\bf s}_{k_l+j_l}
\end{bmatrix}=\begin{bmatrix}
{\bf s}_{0} &\cdots & {\bf s}_{m-1}\\
{\bf s}_{1} &\cdots & {\bf s}_{m}\\
\vdots & & \vdots\\
{\bf s}_{j_l} &\cdots & {\bf s}_{j_l+m-1}
\end{bmatrix}
\begin{bmatrix}
C^{k_1}\cdot \begin{bmatrix} I_p\\0_p\\\vdots\\0_p
\end{bmatrix} & \cdots &
C^{k_l}\cdot \begin{bmatrix} I_p\\0_p\\\vdots\\0_p
\end{bmatrix}
\end{bmatrix}
\end{equation*} 
Then it follows that
\begin{equation*}
\rank \wtilde H_l\leq \rank \begin{bmatrix}
{\bf s}_{k_1} & \cdots & {\bf s}_{k_l} \\
{\bf s}_{k_1+1} & \cdots & {\bf s}_{k_l+1}\\
\vdots & & \vdots\\
{\bf s}_{k_1+j_l}& \cdots & {\bf s}_{k_l+j_l}
\end{bmatrix}\leq \rank \begin{bmatrix}
{\bf s}_{0} &\cdots & {\bf s}_{m-1}\\
{\bf s}_{1} &\cdots & {\bf s}_{m}\\
\vdots & & \vdots\\
{\bf s}_{j_l} &\cdots & {\bf s}_{j_l+m-1}
\end{bmatrix}=mp,
\end{equation*}
which means that $\det\wtilde H_l=0$.
\end{proof}

The following theorem describes Hurwitz matrix polynomials via the block-Hankel minors and quasiminors
built from right Markov parameters.

\begin{theorem}\label{TP}
    Let $F(z)\in\mathbb C[z]^{p\times p}$ be monic of degree $n=2m$ or $n=2m+1$. Assume that
    $\mathscr S$ is the related Hermitian matrix sequence of right Markov parameters (of first
    type for~$n=2m+1$ when the even part $F_e(z)$ is regular). Then $F(z)$ is a monic Hurwitz
    matrix polynomial if and only if the following statements are true simultaneously:
\begin{enumerate}
\item[{\rm (i)}]   all contiguous block-Hankel quasiminors of $H^{\langle \mathscr S\rangle}_{\infty}$ of order $\leq m$ are positive definite.
\item[{\rm (ii)}] all block-Hankel minors of $H^{\langle \mathscr S\rangle}_{\infty}$ of order $>m$ are $0$.
\item[{\rm (iii)}] If $n=2m+1$, the first element of~$\mathscr S$
    satisfies~${\bf s}_{-1}\succ 0$.
\end{enumerate}
\end{theorem}

\begin{proof}
    We only give a proof for the even case $n=2m$. The ``if'' implication is obvious due to
    Theorem \ref{ThmHurwitzMatrixHurwitzOdd}.
    
    The ``'only if'' implication: If~$F(z)$ is Hurwitz stable, (ii) and (iii) are immediate
    consequences of Proposition \ref{VanishHankelMinors} and
    Theorem~\ref{ThmHurwitzMatrixHurwitzOdd}, respectively. Moreover,
    Theorem~\ref{ThmHurwitzMatrixHurwitzOdd} implies $H_{m-1}^{\langle\mathscr S\rangle}$,
    ${H}^{\langle \mathscr S\rangle}_{1,m-1}\succ 0$, and hence
    ${H}^{\langle \mathscr S\rangle}_{k,m-1}\succ 0$ for all~$k\in\N_0$ by~\eqref{H2kCHC}. So,
    the inertia additivity formula for the Schur complement yields~(i).
  \end{proof}

\begin{remark}
    On account of Remark \ref{ReConnectionSLMPandSRMP}, the right Markov parameters in
    Proposition \ref{VanishHankelMinors} and Theorem \ref{TP} may be replaced by the left Markov
    parameters: the corresponding results also hold true.
\end{remark}

It follows from Theorem \ref{TP} that all contiguous block-Hankel minors of order $\leq m$
 are real and positive. This motivated us to seek the block-Hankel total
positivity of~$H^{\langle \mathscr S\rangle}_{\infty}$ as simultaneous satisfaction of the
following:
\begin{enumerate}
\item[(a)] the conditions (i)--(ii) of Theorem \ref{TP} hold, and
\item[(b)] all block-Hankel minors of order~$\leq m$ are positive  real.
\end{enumerate} However, this new property does not help in constructing a direct matrix
generalization of \cite[Theorem 20, Chapter XV]{Gan}. More specifically, the answer to the
following question:

\emph{Is Hurwitz stability of $F(z)$ equivalent to the condition (iii) of Theorem~\ref{TP} and
    the block-Hankel total positivity of~$H^{\langle \mathscr S\rangle}_{\infty}$?}\\
is unfortunately negative, as is seen from the following counterexamples.

Consider the real monic matrix polynomial 
\[
    F(z)=
    \begin{bmatrix}
        x^4+3 x^3+19 x^2+19 x+60 & -\frac{5 x^3}{2}-14 x^2-19 x-56 \\[2pt]
        -\frac{5 x^3}{2}+12 x^2+124 x+24 & x^4+\frac{57 x^3}{4}-9 x^2-\frac{221 x}{2}-22
    \end{bmatrix}
\]
of degree~$4$, whose right Markov parameters
\begin{gather*}
    {\bf s}_{0}=
    \begin{bmatrix}
        3 & -\frac{5}{2} \\[2pt]
        -\frac{5}{2} & \frac{57}{4}
    \end{bmatrix}
    ,\quad
    {\bf s}_{1}=
    \begin{bmatrix}
        8 & -\frac{1}{2} \\[2pt]
        -\frac{1}{2} & \frac{69}{4}
    \end{bmatrix}
    ,\quad
    {\bf s}_{2}=
    \begin{bmatrix}
        26 & \frac{11}{2} \\[2pt]
        \frac{11}{2} & \frac{101}{4}
    \end{bmatrix}
    ,\\
    {\bf s}_{3}=
    \begin{bmatrix}
        92 & \frac{47}{2} \\[2pt]
        \frac{47}{2} & \frac{189}{4}
    \end{bmatrix}
    ,\quad
    {\bf s}_{4}=
    \begin{bmatrix}
        338 & \frac{155}{2} \\[2pt]
        \frac{155}{2} & \frac{437}{4} 
    \end{bmatrix}
\end{gather*}
are Hermitian (here real symmetric). Both block Hankel matrices~$[{\bf s}_{i+j}]_{i,j=0}^1$ and
$[{\bf s}_{i+j+1}]_{i,j=0}^1$ are positive definite, so~$F(z)$ is a Hurwitz matrix polynomial.
Nonetheless,
\[
    \begin{vmatrix}
        {\bf s}_0&{\bf s}_2\\
        {\bf s}_1&{\bf s}_3
    \end{vmatrix}
    =-\frac{3}{4}
    ,\quad
    \begin{vmatrix}
        {\bf s}_0&{\bf s}_3\\
        {\bf s}_1&{\bf s}_4
    \end{vmatrix}
    =
    -\frac{581}{4}
    ,\quad
    \begin{vmatrix}
        {\bf s}_1&{\bf s}_3\\
        {\bf s}_2&{\bf s}_4
    \end{vmatrix}
    =
    -18.
\]

\begin{claim}\label{ClaimNonPositive} Let $F(z)\in\mathbb C[z]^{p\times p}$ be a monic Hurwitz matrix polynomial of
    degree~$n=2m$ or~$n=2m+1$. Assume that~$\mathscr S$ is the related Hermitian matrix sequence
    of left or right Markov parameters (of first type for~$n=2m+1$ if the even part of $F(z)$ is
    regular). Then non-contiguous block-Hankel minors
    of~$H^{\langle \mathscr S\rangle}_{\infty}$ of order~$\le m$ may have non-positive real
    values.
\end{claim}

Now, consider the following monic matrix polynomial of degree~$6$:
\begin{align*}
    F(z)
  ={}&
    \begin{bmatrix}
        1&0\\0& 1
    \end{bmatrix} z^6 +
    \begin{bmatrix}
        8& 3 \i\\-3 \i& 9
    \end{bmatrix} z^5 +
    \frac{17 - 32\i}{1313}\Bigg(\!
    \begin{bmatrix}
        150 + 300 \i& 12 \i\\
        64 + 40 \i& 190 + 340 \i
    \end{bmatrix} z^4
  \\
    &+
    \begin{bmatrix}
        587 + 1664 \i& -1071 + 570 \i\\ 1425 - 186 \i & 1372 + 2356 \i
    \end{bmatrix} z^3 +
    \begin{bmatrix}
        331 + 676 \i & 36 \i\\
        412 + 220 \i & 551 + 956 \i
    \end{bmatrix} z^2
  \\
    &+
    \begin{bmatrix}
        -89 + 2464 \i & -2313 + 837 \i \\ 3467 + 287 \i & 2587 + 4288 \i
    \end{bmatrix} z +
    \begin{bmatrix}
        198 + 408 \i & 24 \i\\ 348 + 180 \i & 378 + 648 \i
    \end{bmatrix}
    \!\Bigg)
\end{align*}
with Hermitian right Markov parameters:
\begin{gather*}
    {\bf s}_{0}=
    \begin{bmatrix}
      8 & 3 \i \\
      -3 \i & 9
  \end{bmatrix}
  ,\
  {\bf s}_{1}=
  \begin{bmatrix}
      29 & 3 \\
      3 & 22 
  \end{bmatrix}
  ,\
  {\bf s}_{2}=
  \begin{bmatrix}
      145 & 21-24 \i \\
      21+24 \i & 100
  \end{bmatrix}
  ,
  \\
  {\bf s}_{3}=
  \begin{bmatrix}
    839 & 153-186 \i \\
    153+186 \i & 592
  \end{bmatrix}
  ,\
  {\bf s}_{4}=
  \begin{bmatrix}
    5173 & 1185-1212 \i \\
    1185+1212 \i & 3784
  \end{bmatrix}
  ,
  \\
  {\bf s}_{5}=
  \begin{bmatrix}
      32879 & 9273-7530 \i \\
      9273+7530 \i & 24832 
  \end{bmatrix}
  .
\end{gather*}
Since both block Hankel matrices $[{\bf s}_{i+j}]_{i,j=0}^2$ and $[{\bf s}_{i+j+1}]_{i,j=0}^2$
are positive definite, $F(z)$ is a Hurwitz matrix polynomial. However, some non-contiguous
block-Hankel minors of~$[{\bf s}_{i+j}]_{i,j=0}^\infty$ of order~$\le \deg F=3$ have complex
values:
\begin{align*}
    \begin{vmatrix}
        {\bf s}_0&{\bf s}_2\\
        {\bf s}_1&{\bf s}_3
    \end{vmatrix}
                   &= 3323095 - 24840 \i
    ,
&    \begin{vmatrix}
        {\bf s}_0&{\bf s}_3\\
        {\bf s}_1&{\bf s}_4
    \end{vmatrix}
                   &=152111099 - 2414520 \i
    ,
\\    \begin{vmatrix}
        {\bf s}_1&{\bf s}_3\\
        {\bf s}_2&{\bf s}_4
    \end{vmatrix}
                   &=327769380 - 2969280 \i
    ,
&    \begin{vmatrix}
        {\bf s}_0&{\bf s}_1&{\bf s}_2\\
        {\bf s}_2&{\bf s}_3&{\bf s}_4\\
        {\bf s}_3&{\bf s}_4&{\bf s}_5
    \end{vmatrix}
                   &=(5859396 - 3456 \i)\cdot 10^3.
\end{align*}

\begin{claim} Under the conditions of Claim \ref{ClaimNonPositive}, non-contiguous block-Hankel
    minors of $H^{\langle \mathscr S\rangle}_{\infty}$ of order~$\le m$ may be not real.
\end{claim}

\section*{Acknowledgement}
We are grateful to Bernd Kirstein, the supervisor of the first author, for motivation to study
this interesting topic. We would also like to thank Conrad M\"adler for insightful discussions.
The second author is grateful to the German Research Foundation~(DFG) for the financial support
(research fellowship DY 133/1-1).
\appendix

\section{Greatest common divisors of matrix polynomials}
\label{sec:great-comm-divis}

\begin{definition}
A matrix
polynomial~$F(z)\in \mathbb C[z]^{p\times p}$ is called
\emph{unimodular} if $\det F(z)$ never vanishes in~$\mathbb C$.
\end{definition}

A GRCD/GLCD of two matrix
polynomials is only unique up to multiplication by a unimodular matrix
polynomial.

\begin{proposition}[{\cite[pp. 377-378]{Kai}}]\label{ReGLCDInertia}
 Let $F(z)$, $\wtilde F(z)\in \mathbb C[z]^{p\times p}$ and let $F_1(z)$ be a GLCD
    (resp.\ GRCD) of $F(z)$ and $\wtilde F(z)$. Then
    $F_2(z)\in \mathbb C[z]^{p\times p}$ is a GLCD (resp.\ GRCD)
    of $F(z)$ and $\wtilde F(z)$ if and only if there exists a unimodular matrix
    polynomial $W(z)\in \mathbb C[z]^{p\times p}$ such that 
    \begin{align*}
   & F_1(z)=F_2(z)W(z)\\
   & (\mbox{resp.\ } F_1(z)=W(z)F_2(z)).
   \end{align*}
\end{proposition}

Lemma 6.3-3 of~\cite{Kai} provides a sufficient way to obtain a GRCD/GLCD. We give a completion
by clarifying the necessity.

\begin{proposition}\label{ProEquivalentGCD}
Let $F(z)$, $F_1(z)$ and $\wtilde F(z) \in \mathbb C[z]^{p\times p}$.
 \begin{enumerate}
 \item[{\rm (i)}] $F_1(z)$ is a GRCD  of $F(z)$ and $\wtilde F(z)$ if and only if there exists a unimodular matrix
     polynomial~$U_R(z)\in \mathbb C[z]^{2p\times 2p}$ and a matrix polynomial
     $F_1(z)\in \mathbb C[z]^{p\times p}$ such that
     \begin{equation}\label{EquivalentGCDI}
         U_R(z)\cdot
         \begin{bmatrix}
             F(z)\\
             \wtilde F(z)
         \end{bmatrix}
         =\begin{bmatrix}
             F_1(z)\\
             0_{p}
         \end{bmatrix}.
     \end{equation}

 \item[{\rm (ii)}] $F_1(z)$ is a GLCD  of $F(z)$ and $\wtilde F(z)$ if and only if there exists a unimodular matrix
     polynomial~$U_L(z)\in \mathbb C[z]^{2p\times 2p}$ such that
     \begin{equation*}
         \begin{bmatrix}
             F(z),
             \wtilde F(z)
         \end{bmatrix}\cdot
         U_L(z)=\begin{bmatrix}
             F_1(z),
             0_{p}
         \end{bmatrix}.
     \end{equation*}

 \end{enumerate}
\end{proposition}

\begin{proof}
      To obtain~``only if'' implication of (i), transform $\begin{bmatrix} F(z)\\\wtilde F(z)
    \end{bmatrix}$ into the Hermitian form, see~\cite[Theorem 6.3-2]{Kai}:
    in particular, there exist a unimodular matrix polynomial
    $\wtilde U_R(z)\in \mathbb C[z]^{2p\times 2p}$ and
    some~$F_2(z)\in \mathbb C[z]^{p\times p}$ such that
    \begin{equation}\label{EquivalentGCDII}
        \wtilde U_R(z)\cdot
        \begin{bmatrix}
            F(z)\\
            \wtilde F(z)
        \end{bmatrix}
        =\begin{bmatrix}
            F_2(z)\\
            0_{p}
        \end{bmatrix}.
    \end{equation}
    Then due to the ``if'' implication, $F_2(z)$ is a
   GRCD of~$F(z)$ and~$\wtilde F(z)$. Suppose that~$F_1(z)$ is a GRCD of~$F(z)$
    and $\wtilde F(z)$. As is seen from Proposition \ref{ReGLCDInertia}, there
    exists a unimodular~$W(z)\in \mathbb C[z]^{p\times p}$ such that
    \begin{equation*}
        W(z)F_2(z)=F_1(z).
    \end{equation*}
    Substituting
    \begin{equation*}
        U_R(z)\coloneqq\mbox{diag}(W(z),I_p) \wtilde U_R(z)
    \end{equation*}
    into~\eqref{EquivalentGCDII} then gives~\eqref{EquivalentGCDI}.

    The validity of~(ii)  follows analogously to the proof of~(i)
.
\end{proof}

Using Proposition \ref{ProEquivalentGCD}, we have

\begin{remark}\label{ProGRCDGLCD}
Suppose that~$F(z)$, $F_1(z)$ and $\wtilde F(z)\in \mathbb C[z]^{p\times p}$. Then $F_1(z)$ is a GRCD of $F(z)$ and
    $\wtilde F(z)$ if and only if $F_1^{\vee}(z)$ is a GLCD of $F^{\vee}(z)$ and~$\wtilde F^{\vee}(z)$.
\end{remark}

The next proposition shows that there is a relation between GRCDs or
GLCDs of matrix polynomials and that of their transformations.

\begin{proposition}\label{ProGCDUnimodular}
Let~$F(z), \wtilde F(z)$ and $F_1(z)\in \mathbb C[z]^{p\times p}$ and let~$U(z)\in \mathbb C[z]^{2p\times 2p}$ be unimodular.
 Define
    $E(z)$, $\wtilde E(z)\in \mathbb C[z]^{p\times p}$
    by:
\begin{align*}
\begin{bmatrix}
E(z)\\
\wtilde E(z)
\end{bmatrix}&\coloneqq
U(z)\cdot \begin{bmatrix}
F(z)\\
\wtilde F(z)
\end{bmatrix}\\
(\text{resp.\ }\  \begin{bmatrix}
E(z),
\wtilde E(z)
\end{bmatrix}& \coloneqq
\begin{bmatrix}
F(z),
\wtilde F(z)
\end{bmatrix}\cdot U(z)).
\end{align*}
Then $F_1(z)$ is a GRCD (resp.\ GLCD) of $F(z)$ and $\wtilde F(z)$ if and only if $F_1(z)$ is a
GRCD (resp.\ GLCD) of~$E(z)$ and~$\wtilde E(z)$.
\end{proposition}

\begin{proof} We only give a proof in the case for GRCD, which can be converted to the case for~GLCD due to Remark \ref{ProGRCDGLCD}.

    Suppose that~$F_1(z)$ is a GRCD of~$F(z)$ and~$\wtilde F(z)$. By
    Proposition~\ref{ProEquivalentGCD}, there exists a unimodular matrix polynomial
    $U_R(z)\in \mathbb C[z]^{2p\times 2p}$ such that~\eqref{EquivalentGCDI} holds.
    Let, for $z\in \mathbb C$,
\begin{equation*}
\wtilde U(z)\coloneqq U_R(z) (U(z))^{-1},
\end{equation*}
which is unimodular. It follows
from~\eqref{EquivalentGCDI} that
\begin{equation}\label{U_RE}
\wtilde U(z) \cdot \begin{bmatrix} E(z)\\ \wtilde E(z) \end{bmatrix}=\begin{bmatrix} F_1(z)\\ 0_{p}
\end{bmatrix}.
\end{equation}
So, Proposition~\ref{ProEquivalentGCD} yields that~$F_1$ is a GRCD of~$E$ and~$\wtilde E$.

Conversely, suppose that $F_1$ is a GRCD of~$E$ and~$\wtilde E$. By
Proposition~\ref{ProEquivalentGCD}, there exists a
unimodular~$\wtilde U(z) \in \mathbb C[z]^{2p\times 2p}$ such
that~\eqref{U_RE} holds. Then the matrix
polynomial~$U_R(z)\coloneqq \wtilde U(z) U(z)$ is unimodular and satisfies
\[
    U_R(z) \cdot \begin{bmatrix} F(z)\\ {\wtilde F}(z) \end{bmatrix} =
    \wtilde U(z) \cdot \begin{bmatrix} E(z)\\ {\wtilde E}(z) \end{bmatrix} =
    \begin{bmatrix} F_1(z)\\ 0_{p}
    \end{bmatrix},
\]
which implies that~$F_1(z)$ is a GRCD of~$F(z)$ and~$\wtilde F(z)$.
\end{proof}

\bibliographystyle{amsplain}

\end{document}